\documentclass[12pt]{article}

\usepackage{amsmath,amsthm,amsfonts,amssymb,amscd}

\allowdisplaybreaks[4]
\newtheorem {Lemma}{Lemma}[section]
\newtheorem {Theorem} {Theorem}[section]

\newtheorem {Corollary}{Corollary}[section]

\numberwithin{equation}{section}

\allowdisplaybreaks [4]
\begin{document}

\title{ Distance spectral radius of uniform hypergraphs}

\author{Hongying Lin\footnote{ E-mail:
lhongying0908@126.com}, Bo Zhou\footnote{Corresponding author. E-mail:
zhoubo@scnu.edu.cn}\\[1.5mm]
School of Mathematical sciences, South China Normal University, \\
Guangzhou 510631, P. R. China }

\date{}
\maketitle

\begin{abstract}
We study the effect of three types of graft transformations to increase or decrease the distance spectral radius of uniform hypergraphs, and  we determined the unique $k$-uniform hypertrees with  maximum, second  maximum, minimum and second minimum  distance spectral radius, respectively.
\\ \\
{\bf AMS classification:} 05C50, 05C65, 15A18\\ \\
{\bf Key words:}  distance spectral radius,  uniform hypergraph, uniform hypertree,  distance matrix, graft transformation
\end{abstract}

\section{Introduction}

A hypergraph $G$ consists of a vertex set $V(G)$ and an edge set $E(G)$,
where $V(G)$ is nonempty, and each edge $e\in E(G)$ is a nonempty subset of $V(G)$, see \cite{Be}. The order of $G$ is
$|V(G)|$. 
For an integer $k\ge 2$, we say that a hypergraph $G$ is $k$-uniform if every edge has size $k$.
A (simple)
graph is a $2$-uniform hypergraph. The degree of a vertex $v$ in $G$, denoted by $d_G(v)$, is the
number of edges of $G$ which contain $v$. 

For $u, v\in V(G)$, a walk  from $u$ to $v$ in $G$ is defined to be a sequence of  vertices and edges
$(v_0,e_1,v_1,\dots,v_{p-1},e_p, v_p)$ with $v_0=u$ and $v_p=v$
edge $e_i$ contains vertices
$v_{i-1}$ and $v_i$, and  $v_{i-1}\ne v_i$ for $i=1,\dots,p$. The value $p$ is the length of this walk.
A path is a walk with all $v_i$ distinct and all $e_i$ distinct.
A cycle is a walk containing at least
two edges, all $e_i$ are distinct and all $v_i$ are distinct except $v_0=v_s$.
A vertex $u\in V(G)$ is viewed as a path (from $u$ to $u$) of length $0$.
If there is a path from $u$ to $v$ for any $u,v\in V(G)$, then we say that $G$ is connected.
A hypertree is a connected hypergraph with no cycles. Note that a
$k$-uniform hypertree with $m$ edges always has order $1+(k-1)m$.

Let $G$ be a $k$-uniform hypergraph with $V(G)=\{v_1,\dots ,v_n\}$ and edge set $E(G)$, where $2\le k\le n$.
For $u,v\in V(G)$, the distance between $u$ and $v$ is the length of a shortest path from $u$ to $v$ in $G$, denoted by  $d_G(u,v)$.  In particular, $d_G(u,u)=0$.
The diameter of $G$ is the maximum distance between all vertex pairs of $G$.
The distance matrix of $G$ is the $n\times n$ matrix
$D(G)=(d_G(u,v))_{u,v\in V(G)}$. The eigenvalues of $D(G)$ are called the distance eigenvalues of $G$. Since $D(G)$ is real and
symmetric, the distance eigenvalues of $G$ are real. The distance
spectral radius of $G$, denoted by $\rho(G)$, is the largest
distance eigenvalue of $G$. Note that  $D(G)$ is a
irreducible nonnegative matrix. The Perron-Frobenius theorem implies that $\rho(G)$
is simple, positive, and there is a unique positive unit eigenvector $x(G)$
corresponding to $\rho(G)$, which is called the distance Perron
vector of $G$, denoted by $x(G)$.

The study of distance eigenvalues of $2$-uniform hypergraphs (ordinary graphs) dates back to the
classical work of Graham and Pollack \cite{GP}, Graham and Lov\'asz
\cite{GL} and Edelberg et al.~\cite{EGG}.
Ruzieh and Powers \cite{SuPo90} showed that among  connected
$2$-uniform hypergraphs of order $n$, the path $P_n$ is the unique graph with
maximum distance spectral radius.
Stevanovi\'c and Ili\'c \cite{SI}
showed that among  trees of order $n$, the star $S_n$ is the unique tree with
minimum distance spectral radius.
Nath and Paul \cite{NP} determined the unique trees with maximum  distance spectral radius
among trees with fixed  matching number. For more details on distance eigenvalues and especially on distance spectral radius of $2$-uniform hypergraphs, one may refer to the recent survey of Aouchiche and Hansen \cite{AH-2} and referees therein. Sivasubramanian \cite{Si} gave a formula for the inverse of a few $q$-analogs of the distance matrix of a $3$-uniform hypertree.


For a $k$-uniform hypertree $G$  with $V(G)=\{v_1, \dots, v_n\}$,  if $E(G)=\{e_1, \dots, $ $ e_m\}$, where $e_i=\{v_{(i-1)(k-1)+1}, \dots, v_{(i-1)(k-1)+k}\}$ for $i=1, \dots, m$, then we call $G$ a $k$-uniform loose path, denoted by $P_{n,k}$.

%
%

For a $k$-uniform hypertree $G$ of order $n$, if there is a disjoint partition of the vertex set $V(G)=\{u\}\cup V_1\cup\dots \cup V_m$ such that  $|V_1|=\dots=|V_m|=k-1$, and $E(G)=\{\{u\}\cup V_i: 1\leq i\leq m\}$, then we call $G$ is a hyperstar (with center $u$), denoted by $S_{n,k}$. 
In particular, $S_{1,k}$ is a hypergraph with a single vertex and  $S_{k,k}$ is a hypergraph with a single edge.

In this paper, we study the effect of three types of graft transformations to increase or decrease the distance spectral radius of $k$-uniform hypergraphs. As applications, we show that $P_{n,k}$ and $S_{n,k}$ are  the unique  $k$-uniform hypertrees with  maximum and minimum distance spectral radius, respectively, and we also determine the unique  $k$-uniform hypertrees with second maximum and second minimum distance spectral radius, respectively.

\section{Preliminaries}

Let $G$ be a $k$-uniform hypergraph with $V(G)=\{v_1,\dots,v_n\}$. A column
vector $x=(x_{v_1},\dots, x_{v_n})^\top\in \mathbb{R}^n$ can be
considered as a function defined on $V(G)$ which maps vertex $v_i$
to $x_{v_i}$, i.e., $x(v_i)=x_{v_i}$ for $i=1,\dots,n$. Then
\begin{equation*}
x^\top D(G)x=\sum_{\{u,v\}\subseteq V(G)} 2d_G(u,v)x_ux_v,
\end{equation*}
and
$\rho$ is a distance eigenvalue with corresponding eigenvector
$x$ if and only if $x \neq 0$ and for each $u\in V(G)$,
\begin{equation*}
\rho
x_u=\sum_{v\in V(G)}d_G(u,v)x_v.
\end{equation*}
The above  equation is called the $(\rho,x)$-eigenequation of $G$ (at $u$).
For a unit column vector $x\in\mathbb{R}^n$ with at least one nonnegative entry, by Rayleigh's principle, we
have
\begin{equation*}
\rho(G)\ge x^{\top}D(G)x
\end{equation*}
with equality if and only if $x=x(G)$.

For $u\in V(G)$, let $G-u$ be the sub-hypergraph of $G$ obtained by deleting $u$ and all edges containing $u$.
We remark that in the literature this is sometimes denoted by strongly deleting the vertex $u$.
For $e\in E(G)$, let  $G-e$ be the sub-hypergraph of $G$ obtained by deleting $e$.

For $X\subseteq V(G)$ with $X\ne\emptyset$, let $G[X]$ be the sub-hypergraph induced by $X$, i.e., $G[X]$ has vertex set $X$ and edge
set $\{e\cap X: e\in E(G)\}$, 
and  let $\sigma_G(X)$ be the sum of the entries of the distance Perron vector of $G$
corresponding to the vertices in $X$.

Let $G$ be a $k$-uniform hypergraph with $u,v\in V(G)$ and $e_1,\dots,e_r\in E(G)$  such that $u\notin e_i$ and $v\in e_i$ for $1\leq i\leq r$.
Let  $e'_i=(e_i\setminus \{v\})\cup\{u\}$ for $1\leq i\leq r$.
Let $G'$ be the hypergraph with $V(G')=V(G)$ and  $E(G')=(E(G)\setminus \{e_1,\dots, e_r\} )\cup\{e'_1,\dots,e'_r\}$.
Then we say that $G'$ is obtained from $G$ by moving edges $e_1,\dots,e_r$ from $v$ to $u$.

A path $P=( v_0,e_1,v_1,\dots,v_{p-1},e_p, v_p)$ in a $k$-uniform hypergraph $G$ is called a pendant path at $v_0$, if $d_G(v_0)\ge 2$,  $d_G(v_i)=2$ for $1\leq i\leq p-1$, $d_G(v)=1$ for $v\in e_i\setminus\{v_{i-1}, v_i\}$ with $1\leq i\leq p$,
and $d_G(v_p)=1$.
An edge $e=\{w_1,\dots,w_k\}$ in $G$ is called a pendant edge at $w_1$ if $d_G(w_1)\geq 2$, $d_G(w_i)=1$ for $2\leq i\leq k$.

If $P$ is a pendant path of a hypergraph $G$ at $u$, we say $G$ is obtained from $H$ by attaching a pendant path $P$ at $u$ with $H=G[V(G)\setminus(V(P)\setminus\{u\})]$. If $P$ is a pendant path of length one of $G$ at $u$, then we also say that $G$ is obtained from $H$ by attaching a pendant edge at $u$.

A component of a hypergraph $G$
is a maximal connected sub-hypergraph of $G$.

For vertex-disjoint hypergraphs $G_1$ and $G_2$, let $G_1\cup G_2$ the be union of $G_1$ and $G_2$, i.e, $V(G_1\cup G_2)=V(G_1)\cup V(G_2)$ and $E(G_1\cup G_2)=E(G_1)\cup E(G_2)$.

\section{The effect of graft transformations on distance spectral radius}

In this section we propose three types of graft transformations for a $k$-uniform hypergraph, and  consider the effect of them to increase or decrease the distance spectral radius.

Let $G$ be a connected $k$-uniform hypergraph with $ |E(G)|\geq 1$. For $u\in V(G)$, and  positive integers $p$ and $q$, let  $G_{u}(p,q)$ be a $k$-uniform hypergraph obtained from $G$ by attaching two pendant paths of lengths $p$ and $q$ at $u$, respectively,  and  $G_{u}(p,0)$ be a $k$-uniform hypergraph obtained from $G$ by attaching a pendant path of length $p$ at $u$.

\begin{Theorem} \label{moving1} Let $G$ be a connected $k$-uniform hypergraph  $| E(G)|\geq 1$ and  $u\in V(G)$.
For integers $ p\geq q\geq 1$, $\rho(G_u(p,q))<\rho(G_u(p+1,q-1))$.
\end{Theorem}

\begin{proof} Let $H=G_{u}(p,q)$.
Let $P=( u,e_1,u_1,\dots, u_{p-1},e_p,u_p)$ and $Q=(u,e'_1,v_1,$ $ \dots,v_{q-1},e'_q,v_q)$ be the pendant paths of $H$ at $u$ of lengths $p$ and $q$, respectively.
Let $e_{u}^1,\dots,e_{u}^{t}$ be all the edges of $G$ containing $u$, where $t\geq 1$.

Suppose that $\sigma_H (V(P))\geq \sigma_H (V(Q))$. Let $H'$ be the hypergraph obtained from $H$ by moving edges $e_{u}^1,\dots,e_{u}^{t}$ from $u$ to $v_1$.
It is easily seen that $H'\cong G_u(p+1,q-1)$.
As we pass from $H$ to $H'$, the distance between a vertex of $V(G)\setminus \{u\}$ and a vertex of $V(P)$ is increased by $1$,  the distance between a vertex of $V(G)\setminus \{u\}$ and $V(Q)\setminus (e'_1\setminus\{v_1\})$ is decreased by $1$,
and the distance between any other vertex
pair remains unchanged. Let $x=x(H)$. Then
\begin{eqnarray*}
&&\frac{1}{2}(\rho(H')- \rho(H))\\
&\ge&\frac{1}{2}x^{\top}(D(H')-D(H))x\\
&=&\sigma_H(V(G)\setminus \{u\})\left[\sigma_H(V(P))- (\sigma_H(V(Q))-\sigma_H(e'_1\setminus\{v_1\}))\right]\\
&=&\sigma_H(V(G)\setminus \{u\})\left[\sigma_H(V(P))-\sigma_H(V(Q))+\sigma_H(e'_1\setminus\{v_1\})\right]\\
&\ge &\sigma_H(V(G)\setminus \{u\})\sigma_H(e'_1\setminus \{v_1\})\\
&>&0,
\end{eqnarray*}
and thus $\rho(H)<\rho(H')$.

Suppose that $\sigma_H (V(P))<\sigma_H (V(Q))$.
Let $H''$ be the hypergraph obtained from $H$ by moving edges $e_{u}^1,\dots,e_{u}^{t}$ from $u$ to $u_1$.
It is easily seen that $H''\cong G_{u}(p-1,q+1)$.
By similar argument as above, we have that $\rho(H)<\rho(H'')$.

Thus we have proved that
\begin{equation}
\label{two}
\rho(G_u(p,q))<\max\{\rho(G_u(p-1,q+1)), \rho(G_u(p+1,q-1))\}.
\end{equation}
If $p=q$, then the result follows. Suppose that
$p>q$. If $\rho(G_u(p,q))<\rho(G_u(p-1,q+1))$, then by using (\ref{two}) repeatedly, we have
\begin{eqnarray*}
\rho(G_u(p,q))&\le &\rho\left(G_u\left(\left\lceil\frac{p+q+1}{2}\right\rceil, \left\lfloor\frac{p+q-1}{2}\right\rfloor\right)\right)\\
&<&\rho\left(G_u\left(\left\lfloor\frac{p+q-1}{2}\right\rfloor, \left\lceil\frac{p+q+1}{2}\right\rceil\right)\right),
\end{eqnarray*}
which is impossible.
Thus $\rho(G_u(p,q))<\rho(G_u(p+1,q-1))$.
\end{proof}

Let $G$ be a connected $k$-uniform hypergraph with $u,v\in e\in E(G)$.  For positive integers $p$ and  $q$, let   $G_{u,v}(p,q)$ be a $k$-uniform hypergraph obtained from $G$ by attaching a pendant path of length $p$ at $u$ and a pendant path   of length $q$ at $v$, and $G_{u,v}(p,0)$ be a $k$-uniform hypergraph obtained from $G$ by attaching a pendant path of length $p$ at $u$. Let $G_{u,v}(0,q)=G_{v,u}(q, 0)$.

\begin{Theorem} \label{moving2} Let $G$ be a connected $k$-uniform hypergraph  $|E(G)|\geq 2$ and  $u,v\in e\in E(G)$.
Suppose that $G-e$ consists of $k$  components.
For integers $ p,q\geq 1$, $\rho(G_{u,v}(p,q))<\rho(G_{u,v}(p-1,q+1))$ or $\rho(G_{u,v}(p,q))<\rho(G_{u,v}(p+1,q-1))$.
\end{Theorem}

\begin{proof} Let $H=G_{u,v}(p,q)$.
Let $P=(u,e_1,u_1,\dots, u_{p-1},e_p,u_p)$ and
$Q=(v,e'_1,v_1,\dots,v_{q-1},e'_q,v_q)$ be the pendant paths of $H$ at $u$ and
 $v$ of lengths $p$ and $q$, respectively.
Let $e=\{u,v,w_1,\dots, w_{k-2}\}$, $e_1=\{u,u_1,w'_1,\dots, w'_{k-2}\}$ and $e_1'=\{v,v_1,w''_1,\dots, w''_{k-2}\}$.
For $w\in e$, let $G_w$ be the component of $G-e$ containing $w$, and let $t_w$ be the number of edges of $G-e$ containing $w$.
For $w\in e$ with $t_w\ge 1$, let $e_w^1,\dots,e_w^{t_w} $ be the $t_w$ edges of $G-e$ containing $w$. Note that $e_w^1,\dots,e_w^{t_w} \in E(G_w)$ for  $w\in e$. 

Suppose that $\sigma_H (V(P))\geq \sigma_H (V(Q))$. Let $H'$ be the hypergraph obtained from $H$ by moving edges $e_{w_i}^1,\dots,e_{w_i}^{t_{w_i}}$ from $w_i$ to $w''_i$ for all $i$ with $1\leq i \leq k-2$ and $t_i\geq 1$ if $k\geq3$,
moving edges $e_u^1,\dots,e_u^{t_u}$ from $u$ to $v$ if $t_u\geq 1$, and moving edges $e_v^1,\dots,e_v^{t_v}$ from $v$ to $v_1$ if $t_v\geq 1$.
It is easily seen that $H'\cong G_{u,v}(p+1,q-1)$.

For $k\geq 3$, let $I=\{i: 1\leq i\leq k-2,~t_{w_i}\geq 1\}$.
As we pass from $H$ to $H'$,
the distance between a vertex of $V(G_u)\setminus \{u\}$ and a vertex of $V(P)$ is increased by $1$,
and the distance between a vertex of $V(G_u)\setminus \{u\}$ and a vertex of $V(Q)$ is decreased by $1$,
the distance between a vertex of $V(G_v)\setminus \{v\}$ and a vertex of $V(P)\cup (e\setminus\{u\})$ is increased by $1$,
the distance between a vertex of $V(G_v)\setminus \{v\}$ and a vertex of $V(Q)\setminus (e'_1\setminus\{v_1\})$ is decreased by $1$,
for $i\in  I$ with $k\geq 3$, the distance between a vertex of $V(G_{w_i})\setminus \{w_i\}$ and a vertex of $V(P)\cup(e\setminus\{w_i,u,v\})$ is increased by $1$,
the distance between a vertex of $V(G_{w_i})\setminus \{w_i\}$ and $w_i$ is increased by $2$,
the distance between a vertex of $V(G_{w_i})\setminus \{w_i\}$ and a vertex of $V(Q)\setminus \{v,w''_i\}$ is decreased by $1$, and
the distance between a vertex of $V(G_{w_i})\setminus \{w_i\}$ and $w''_i$ is decreased by $2$, and
the distance between any other vertex
pair remains unchanged. Let $x=x(H)$. Note that $\sigma_H (e)-x_u\geq x_v$ and $\sigma_H(e'_1)-x_{v_1}\geq x_v$. Then
\begin{eqnarray*}
&&\frac{1}{2}(\rho(H')- \rho(H))\\
&\ge&\frac{1}{2}x^{\top}(D(H')-D(H))x\\
%
%
&=&\sigma_H(V(G_u)\setminus\{u\})\left(\sigma_H(V(P))-\sigma_H(V(Q))\right)\\
&&+\sigma_H(V(G_v)\setminus\{v\})\left[\sigma_H(V(P))+(\sigma_H(e)-x_u)\right.\\
&&\left. -(\sigma_H(V(Q))-(\sigma_H(e'_1)-x_{v_1}))\right]\\
&&+\sum_{i\in I}\sigma_H(V(G_{w_i})\setminus\{w_i\})\left[\sigma_H(V(P))+(\sigma_H(e)-x_{w_i}-x_u-x_v)+2x_{w_i}\right.\\
&&\left.-(\sigma_H(V(Q))-x_v-x_{w''_i})-2x_{w''_i} \right]\\
&= &\sigma_H(V(G_u)\setminus\{u\})\left(\sigma_H(V(P))-\sigma_H(V(Q))\right )\\
&&+\sigma_H(V(G_v)\setminus\{v\})\left[\sigma_H(V(P))-\sigma_H(V(Q))\right.\\
&&\left. +(\sigma_H (e)-x_u)+(\sigma_H(e'_1)-x_{v_1})\right]\\
&&+\sum_{i\in I}\sigma_H(V(G_{w_i})\setminus\{w_i\})\left[\sigma_H(V(P))-\sigma_H(V(Q))\right.\\
&& \left.+(\sigma_H(e)-x_u)+x_{w_i}-x_{w''_i} \right]\\
&\geq &\sigma_H(V(G_u)\setminus\{u\})\left(\sigma_H(V(P))-\sigma_H(V(Q))\right )\\
&&+\sigma_H(V(G_v)\setminus\{v\})\left(\sigma_H(V(P))-\sigma_H(V(Q))+2x_v\right)\\
&&+\sum_{i\in I}\sigma_H(V(G_{w_i})\setminus\{w_i\})\left(\sigma_H(V(P))-\sigma_H(V(Q))+x_v+x_{w_i}-x_{w''_i} \right)\\
&\geq &2\sigma_H(V(G_v)\setminus\{v\})x_v+\sum_{i\in I}\sigma_H(V(G_{w_i})\setminus\{w_i\})\left(x_v+x_{w_i}-x_{w''_i} \right).
\end{eqnarray*}
By the eigenequations of $H$ at $v,w_i$ and $w''_i$, we have
\begin{eqnarray*}
\rho(H)x_v&=&x_{w_i}+x_{w''_i}+\sum_{w\in V(H)\setminus \{v,w_i,w''_i\}}d_H(v ,w),\\
\rho(H)x_{w_i}&=&x_v+2x_{w''_i}+\sum_{w\in V(H)\setminus \{v,w_i,w''_i\}}d_H(w_i,w),\\
\rho(H)x_{w''_i}&=&2x_{w_i}+x_v+\sum_{w\in V(H)\setminus \{v,w_i,w''_i\}}d_H(w''_i,w).
\end{eqnarray*}
For $w\in V(H)\setminus \{v,w_i,w''_i\}$ with $i\in I$, we have $d_H(v, w)+d_H(w_i ,w)-d_H(w''_i,w)\geq 0$.
Then $\rho(H)(x_v+x_{w_i}-x_{w''_i})\geq -x_{w_i}+3x_{w''_i}$.
Since $(\rho(H)+1)(x_v+x_{w_i}-x_{w''_i})\geq x_v+2x_{w''_i}>0$, we have $x_v+x_{w_i}-x_{w''_i}>0$.
Thus
\begin{equation}\label{eq1}
\begin{array}{lll}
\frac{1}{2}(\rho(H')- \rho(H))
&\ge&\frac{1}{2}x^{\top}(D(H')-D(H))x\\  [2mm]
&\geq &2\sigma_H(V(G_v)\setminus\{v\})x_v\\[2mm]
&&+\sum_{i\in I}\sigma_H(V(G_{w_i})\setminus\{w_i\})\left(x_v+x_{w_i}-x_{w''_i} \right)\\ [2mm]
&\geq& 0.
\end{array}
\end{equation}
It follows that $\rho(H)\le \rho(H')$.

If $I\ne\emptyset$, then $\sum_{i\in I}\sigma_H(V(G_{w_i})\setminus\{w_i\})>0$, and thus from (\ref{eq1}) we have $\rho(H)<\rho(H')$.
Suppose that $I=\emptyset$ and $\rho(H)= \rho(H')$. Then  $\sigma_H(V(G_v)\setminus\{v\})=0$, i.e., $V(G_v)=\{v\}$.
Since $|E(G)|\geq 2$,  we have $V(G_u)\setminus\{u\}\neq \emptyset$. Thus $d_{H}(u,z)<d_{H'}(u,z)$ for $z\in V(G_u)\setminus\{u\}$.
On the other hand, since all inequalities  in (\ref{eq1}) are equalities, we have $x^\top D(H')x=\rho(H')$, and thus
$x(H')=x$, implying that
$\rho(H)x_u=(D(H)x)_u<(D(H')x)_u=\rho(H')x_u$
contradiction. Thus $\rho(H)<\rho(H')$.

%
%

Suppose that $\sigma_H (V(P))<\sigma_H (V(Q))$.
Let $H''$ be the hypergraph obtained from $H$ by moving edges $e_{w_i}^1,\dots,e_{w_i}^{t_{w_i}}$ from $w_i$ to $w'_i$ for all $i$ with $1\leq i \leq k-2$ and  $t_i\geq 1$  if $ k\geq 3$,
moving edges $e_u^1,\dots,e_u^{t_u}$ from $v$ to $u$ if $t_u\geq 1$, and moving edges $e_v^1,\dots,e_v^{t_v}$ from $u$ to $u_1$ if $t_v\geq 1$.
It is easily seen that $H''\cong G_{u,v}(p-1,q+1)$.
By similar argument as above, we have that $\rho(H)<\rho(H'')$.

Thus $\rho(G_{u,v}(p,q))<\rho(G_{u,v}(p-1,q+1))$ or $\rho(G_{u,v}(p,q))<\rho(G_{u,v}(p+1,q-1))$, as desired.
\end{proof}

\begin{Corollary} \label{moving22} Let $G$ be a connected $k$-uniform hypergraph with $|E(G)|\geq 2$ and $u,v\in e\in E(G)$.
Suppose that $G-e$ consists of $k$  components and $d_G(u)=d_G(v)=1$.
For $p\ge q\ge 1$, $\rho(G_{u,v}(p,q))<\rho(G_{u,v}(p+1,q-1))$.
\end{Corollary}

\begin{proof} If $p=q$, then the result follows from Theorem~\ref{moving2}. Suppose that $p>q$.
Suppose that $\rho(G_{u,v}(p,q))<\rho(G_{u,v}(p-1,q+1))$. Since $d_G(u)=d_G(v)=1$, we have
 $G_{u,v}(\lceil\frac{p+q+1}{2}\rceil,\lfloor\frac{p+q-1}{2}\rfloor)\cong G_{u,v}(\lfloor\frac{p+q-1}{2}\rfloor,\lceil\frac{p+q+1}{2}\rceil)$.
Using  Theorem~\ref{moving2} repeatedly, we
have $\rho(G_{u,v}(p,q))\le \rho(G_{u,v}(\lceil\frac{p+q+1}{2}\rceil,\lfloor\frac{p+q-1}{2}\rfloor))<\rho(G_{u,v}(\lfloor\frac{p+q-1}{2}\rfloor,\lceil\frac{p+q+1}{2}\rceil))$, which is impossible.
Thus $\rho(G_{u,v}(p,q))<\rho(G_{u,v}(p+1,q-1))$.
\end{proof}


Let $G$ be a connected $k$-uniform hypergraph with $|E(G)|\geq 2$, and let $e=\{w_1,\dots, w_k\}$ be an edge of $G$, where $d_{G}(w_i)=1$ for $1\leq i\leq k-1$, $d_{G}(w_k)\geq 2$. For $1\le i\le k-1$, let $H_i$ be a connected $k$-uniform hypergraph  with $v_i\in V(H_i)$. Suppose that $G, H_1, \dots, H_{k-1}$ are vertex-disjoint.
For $0\le s\le k-1$, let $G_{e,s}(H_1,\dots, H_{k-1})$  be the hypergraph obtained by identifying $w_{i}$ of $G$ and $v_i$ of $H_i$ for  $s+1\leq i\leq k-1$  and identifying  $w_{k}$ of $G$ and $v_i$ of $H_i$ for all $i$ with $1\leq i\leq s$.
In particular, if $H_i=S_{t_i(k-1)+1,k}$ and $v_i$ is its heart, where $t_i\geq 0$ and $1\leq i\leq k-1$, then we write $G_{e,s}(t_1,\dots, t_{k-1})$ for $G_{e,s}(H_1,\dots,H_{k-1})$.

\begin{Theorem} \label{moving3}
Suppose that  $|E(H_j)|\geq 1$ for some $j$ with $1\leq j\leq k-1$.
Then $\rho (G_{e,0}(H_1,\dots,H_{k-1}))>\rho(G_{e,s}(H_1,\dots,H_{k-1}))$ for $j\leq s \leq k-1$.
\end{Theorem}

\begin{proof} Let $H=G_{e,0}(H_1,\dots,H_{k-1})$. For $1\leq i\leq k-1$, let $t_i$ be the number of edges of  $H-e$ containing $w_i$.
For $1\leq i\leq k-1$ with $t_i\ge 1$, let $e_i^1,\dots, e_i^{t_i}$ be the $t_i$ edges of  $H-e$ containing $w_i$.
Let $H'$ be the hypergraph obtained by moving edges $e_i^1,\dots, e_i^{t_i}$  from $w_i$ to $w_k$ for all $i$ with $1\leq i\leq s$ and $t_i\geq 1$.
It is easily seen that $H'\cong G_{e,s}(H_1, \dots, H_{k-1})$.

Let $I=\{i: 1\leq i\leq s,~t_i\geq 1 \}$.
As we pass from $H$ to $H'$, for $i,l \in I$ with $i\neq l$,
the distance between a vertex of $V(H_i)\setminus \{w_i\}$ and a vertex of $V(H_l)\setminus \{w_l\}$ is decreased by $1$,
the distance between a vertex of $V(H_i)\setminus \{w_i\}$ and a vertex of $(V(G)\setminus e) \cup\{w_k\}$ is decreased by $1$,
the distance between a vertex of $V(H_i)\setminus \{w_i\}$ and $w_i$ is increased by $1$,
and the distance between any other vertex
pair remains unchanged. Let $x=x(H')$. Then
\begin{equation}\label{eq2}
\begin{array}{lll}
&&\frac{1}{2}(\rho(H)- \rho(H'))\\ [2mm]
&\ge&\frac{1}{2}x^{\top}(D(H)-D(H'))x \\ [2mm]
&=&\sum_{i\in I}\sigma_{H'}(V(H_i)\setminus \{w_i\})\left(\sum_{l\in I \atop l>i}\sigma_{H'}(V(H_l)\setminus \{w_l\})\right.\\[2mm]
&&\left.+\sigma_{H'}(V(G)\setminus e)+x_{w_k}-x_{w_i}\vphantom{\sum_{j=i+1}^{s}}\right) \\ [2mm]
&\geq&\sum_{i\in I}\sigma_{H'}(V(H_i)\setminus \{w_i\})\left(\sigma_{H'}(V(G)\setminus e) +x_{w_k}-x_{w_i}\right)\\[2mm]
&\geq&\sigma_{H'}(V(H_t)\setminus \{w_j\})\left(\sigma_{H'}(V(G)\setminus e) +x_{w_k}-x_{w_j}\right).
\end{array}
\end{equation}
Note that there is an edge different from $e$ containing $w_k$. Let $v$ be a vertex different from $w_k$ in such an edge of $G$.
For $w\in V(H)\setminus \{v,w_k,w_j\}$, we have $d_{H'}(v ,w)+d_{H'}(w_k ,w)-d_{H'}(w_j ,w)\geq 0$.
By the eigenequations of $H'$ at $v,w_k$ and $w_j$, we have $\rho(H')(x_v +x_{w_k}-x_{w_j})\geq -x_v +3x_{w_j}$. Since $(\rho(H')+1)(x_v+x_{w_k}-x_{w_j})\geq x_{w_k}+2x_{w_j}>0$, we have $x_v +x_{w_k}-x_{w_j}>0$.
Then $\sigma_{H'}(V(G)\setminus e) +x_{w_k}-x_{w_j}>x_v +x_{w_k}-x_{w_j}>0$.
From~(\ref{eq2}), we have $\rho(H)>\rho(H')$.
\end{proof}

\begin{Corollary} \label{moving33} For $G_{e,0}(t_1,\dots, t_{k-1})$, if $t_j\geq 1$ for some $j$ with  $1\leq j\leq k-1$,
then $\rho (G_{e,0}(t_1,\dots, t_{k-1}))>\rho(G_{e,s}(t_1,\dots,t_{k-1}) $ for $j\leq s \leq k-1$.
\end{Corollary}

\section{Distance spectral radius of uniform hypertrees}

In this section we study the distance spectral radius of $k$-uniform hypertrees using the results in Section~3.

For positive integers $\Delta ,n$ with $1\le \Delta\le \frac{n-1}{k-1}$,
let $B_{n,k}^\Delta$ be the $k$-uniform hypertree obtained from vertex-disjoint  hyperstar $S_{(\Delta-1)(k-1)+1,k}$ with center $u$  and loose path $P_{n-(\Delta-1)(k-1),k}$ with an end vertex $v$ by identifying $u$ and $v$.
In particular,  $B_{n,k}^\Delta \cong P_{n,k}$ if $\Delta =1,2$.

\begin{Theorem}\label{max-hyper2}
Let $T$ be a $k$-uniform hypertree on $n$ vertices with maximum degree $\Delta$, where $1\le \Delta\le \frac{n-1}{k-1}$. Then $\rho(T)\leq \rho(B_{n,k}^\Delta)$ with equality if and only if $T\cong B_{n,k}^\Delta$.
\end{Theorem}

\begin{proof}
It is trivial if $\Delta= 1$. Suppose that $\Delta\geq 2$.
Let $T$ be a $k$-uniform hypertree of order $n$ with maximum degree $\Delta$ having maximum distance spectral radius.

Let $u$ be a vertex of $T$ with degree $\Delta$.

\noindent
{\bf Case~1.} $\Delta\geq 3$.

Suppose that there are at least two vertices of degree at least $3$ in $T$.
Choose a vertex $v$ of degree at least $3$ such that  $d_T(u,v)$ is as large as possible.
Let $T_1,\dots,T_{d_T(v)}$ be the vertex disjoint sub-hypergraphs of $T-v$ such that  $T[V(T_i)\cup\{v\}]$  is a  $k$-uniform hypertree for $1\leq i\leq d_T(v)$. Assume that $u\in V(T_1)$.
If $k=2$, then $T[V(T_i)\cup\{v\}]$ is a pendant path at $v$ for  $2\leq i\leq d_T(v)$.
Suppose that $k\geq 3$ and  $T[V(T_i)\cup\{v\}]$  is not a pendant path at $v$ for some $i$ with $2\leq i\leq d_T(v)$.
Then there is at least one edge  in $ E(T[V(T_i)\cup\{v\}])$ with at least three vertices of degree two. We choose such an edge $e=\{w_1,\dots ,w_k\}$ by requiring that
$d_T(v,w_1)$
is as large as possible, where $d_T(v,w_1)=d_T(v,w_j)-1$ for $2\leq j\leq k$.
Then
there are at least two pendant paths  at different vertices of $e$, say $P=(w_j,e_1,u_1,\dots, u_{p-1},e_p,u_p)$ at $w_j$ and $Q=(w_l,e'_1,v_1,\dots,v_{q-1},e'_q,v_q)$ at $w_l$, where $2\leq j<l\leq k$ and $p,q\geq 1$.
Then $T\cong H_{w_j,w_l}(p,q)$, where $H=T[V(T)\setminus (V(P\cup Q)\setminus\{w_j,w_l\}) ]$. Assume that $p\ge q$.
Note that $d_H(w_j)=d_H(w_l)=1$ and  $T'=H_{w_j,w_l}(p+1,q-1)$ is a $k$-uniform hypertree with maximum degree $\Delta$.
By Corollary~\ref{moving22}, we have $\rho(T)<\rho(T')$, a contradiction.
Thus $T[V(T_i)\cup\{v\}]$ is a pendant path at $v$ for  $2\leq i\leq d_T(v)$ for $k\geq 2$.
Let $l_i$ be the length of the pendant path $T[V(T_i)\cup\{v\}]$ at $v$, where  $2\leq i\leq d_T(v)$ and $l_i\geq 1$.
Then $T\cong G_v(l_2,l_3)$, where $G=T[V(T)\setminus V(T_2\cup T_3)]$. Assume that $l_2\ge l_3$.  Note that $T''=G_v(l_2+1,l_3-1)$ is a $k$-uniform hypertree with  maximum degree $\Delta$.
By Theorem~\ref{moving1}, $\rho(T)<\rho(T'')$, a contradiction.
Thus $u$ is the unique  vertex of degree at least $3$ in $T$.

Let $G_1,\dots,G_\Delta$ be the vertex disjoint sub-hypergraphs of $T-u$ such that $T[V(G_i)\cup\{u\}]$  is a connected $k$-uniform hypergraph for  $1\leq i\leq \Delta$.
By similar argument as above, $T[V(G_i)\cup\{u\}]$  is a pendant path at $u$ for  $1\leq i\leq \Delta$.
Let $l_i$ be the length of the pendant path $T[V(G_i)\cup\{u\}]$ at $u$ for $1\leq i\leq \Delta$.
Suppose that there are at least two pendant paths of length at least two at $u$, say $T[V(G_i)\cup\{u\}]$ and $T[V(G_j)\cup\{u\}]$ are such two paths, where $1\leq i<j\leq \Delta$.
Then $T\cong H_u(l_i,l_j)$, where $H=T[V(T)\setminus V(G_i\cup G_j) ]$. Assume that $l_i\ge l_j$.
Note that  $T^*=H_u(l_i+1,l_j-1)$ is also a $k$-uniform hypertree with maximum degree $\Delta$.
By Theorem~\ref{moving1}, $\rho(T)<\rho(T^*)$, a contradiction.
Thus there is at most one pendant path of length at least $1$, implying that  $T\cong B_{n,k}^\Delta$.

\noindent
{\bf Case~2.} $\Delta=2$.

It is trivial if $k=2$.
Suppose that $k\geq 3$ and  $T\ncong B_{n,k}^2$. Then there is an edge in $T$ with at least three vertices of degree $2$.
Choose such an edge $e=\{w_1,\dots ,w_k\}$ in $ E(T)$  such that $d_T(u,w_1)$ is as large as possible, where $d_T(u,w_1)=d_T(u,w_j)-1$ for $2\leq j\leq k$.
Then there are two pendant paths  at different vertices of $e$, say $P=(w_j,e_1,u_1,\dots, u_{p-1},e_p,u_p)$ at $w_j$ and $Q=(w_l,e'_1,v_1,\dots,v_{q-1},e'_q,v_q)$ at $w_l$, where $2\leq j<l\leq k$ and  $p,q\geq 1$.
Then $T\cong H_{w_j,w_l}(p,q)$, where $H=T[V(T)\setminus (V(P\cup Q)\setminus \{w_j,w_l\}) ]$. Assume that $p\ge q$.
Note that $d_H(w_j)=d_H(w_l)=1$ and  $T'=H_{w_j,w_l}(p+1,q-1)$ is a $k$-uniform hypertree with maximum degree~$2$.
By Corollary~\ref{moving22}, we have $\rho(T)<\rho(T')$, a contradiction.
Thus there are at most two vertices of degree two in each edge, implying that $T\cong B_{n,k}^2$.

Combining Cases~1 and~2, we complete the proof.
\end{proof}

\begin{Theorem}\label{max-hyper1}
Let $T$ be a $k$-uniform hypertree on $n$ vertices, where $\frac{n-1}{k-1}\geq 1$. Then $\rho(T)\leq \rho(P_{n,k})$ with equality if and only if $T\cong P_{n,k}$.
\end{Theorem}
\begin{proof} It is trivial if $\frac{n-1}{k-1}=1,2$. Suppose that $\frac{n-1}{k-1}\geq 3$.
Let $T$ be a $k$-uniform hypertree of order $n$ with maximum distance spectral radius.
Let $\Delta$ be the maximum degree of $T$. Then by Theorem~\ref{max-hyper2},
$T\cong B_{n,k}^\Delta$. Suppose that $\Delta\geq 3$.
By Theorem~\ref{moving1}, we have $\rho(B_{n,k}^\Delta)<\rho(B_{n,k}^{\Delta-1})$, a contradiction.
Then $\Delta=2$, and thus $T\cong B_{n,k}^2\cong P_{n,k}$.
\end{proof}

For $k\ge 3$ and a loose path$P_{n-k+1,k}=(u_0,e_1,u_1,\dots ,e_{\frac{n-k}{k-1}},u_{\frac{n-k}{k-1}})$, where $\frac{n-1}{k-1}\geq 3$,
let 
$F_{n,k}$ be the $k$-uniform hypertree obtained from $P_{n-k+1,k}$  by attaching a pendant edge at
a vertex of degree one in $e_2$. If $\frac{n-1}{k-1}=3$, then $F_{n,k}\cong P_{n,k}$.

\begin{Lemma}\label{hyper1}
Suppose that $\frac{n-1}{k-1}\geq 3$ and $k\geq 3$. Then $\rho(B_{n,k}^3)<\rho(F_{n,k})$.
\end{Lemma}
\begin{proof} If $\frac{n-1}{k-1}=3$, then the result follows from Theorem~\ref{max-hyper1}.

Suppose that $\frac{n-1}{k-1}\geq 4$.
Let $T=F_{n,k}$. Let $v\in e_2$ such that $d_T(v)=1$, and let $e$ be the pendant edge at $v$.
Let $T'$ be the hypergraph obtained from $T$ by moving $e$ from $v$ to $u_1$. Obviously, $T'\cong B_{n,k}^3$.

As we pass from $T$ to $T'$, the distance between a vertex of $e\setminus\{v\}$ and a vertex of $e_1$ is decreased by $1$, the distance between a vertex of $e\setminus\{v\}$ and $v$ is increased by $1$, and the distance between any other vertex pair remained unchanged.
Let $x=x(T')$.
Then
\begin{eqnarray*}
\frac{1}{2}(\rho(T)-\rho(T'))x_v&\geq&\frac{1}{2}x^\top(D(T)-D(T'))x\\
&=&( \sigma_{T'}(e)-x_{v})(\sigma_{T'}(e_1)-x_v).
\end{eqnarray*}
From the eigenequations of $T$ at $u_0,u_1$ and $v$, we have
\begin{eqnarray*}
\rho(T')x_{u_0}&=&x_{u_1}+2x_v+\sum_{i=2}^{k-1}x_{w_i}+\sum_{w\in V(T)\setminus(e_1\cup\{v\})}d_{T'}(u_0,w)x_w,\\
\rho(T')x_{u_1}&=&x_{u_0}+x_v+\sum_{i=2}^{k-1}x_{w_i}+\sum_{w\in V(T)\setminus(e_1\cup\{v\})}d_{T'}(u_1,w)x_w,\\
\rho(T')x_{v}&=&2x_{u_0}+x_{u_1}+\sum_{i=2}^{k-1}2x_{w_i}+\sum_{w\in V(T)\setminus(e_1\cup\{v\})}d_{T'}(v,w)x_w.
\end{eqnarray*}
Note that for $w\in V(T)\setminus(e_1\cup\{v\})$, $d_{T'}(u_0,w)+d_{T'}(u_1,w)-d_{T'}(v,w)\geq0$.
Then $ (\rho(T')+1)(x_{u_0}+x_{u_1}-x_{v})\geq x_{u_1}+2x_v>0$, and thus $\sigma_{T'}(e_1)-x_v>x_{u_0}+x_{u_1}-x_{v}>0$.
It follows that $\rho(T')<\rho(T)$, as desired.
\end{proof}


Let $T$ be a $k$-uniform hypertree of order $n$, where $T\ncong P_{n,k}$. Then $\frac{n-1}{k-1}\geq 3$, and if $\frac{n-1}{k-1}=3$, then $T\cong S_{n,k}$.

Let $F_{n,2}=B^3_{n,2}$.

\begin{Theorem}\label{max-hyper3}
Let $T$ be a $k$-uniform hypertree of order $n$, where $T\ncong P_{n,k}$ and $\frac{n-1}{k-1}\geq4$. Then $\rho(T)\leq \rho(F_{n,k}) $ with equality if and only if $T\cong F_{n,k}$.
\end{Theorem}

\begin{proof} Let $T$ be a $k$-uniform hypertree of order $n$ nonisomorphic to $P_{n,k}$ with maximum distance spectral radius.

Let $\Delta$ be the maximum degree of $T$. Then $\Delta\geq 3$ if $k=2$ and $\Delta\geq 2$ if $k\geq 3$.

Suppose that $\Delta\geq 3$. Then by Theorem~\ref{max-hyper2},
$T\cong B_{n,k}^\Delta$. Suppose that $\Delta\geq 4$. Note that $B_{n,k}^{\Delta-1} \ncong P_{n,k}$.
By Theorem~\ref{moving1}, we have $\rho(B_{n,k}^\Delta)<\rho(B_{n,k}^{\Delta-1})$, a contradiction.
Thus $\Delta=3$. It follows that $T\cong B_{n,k}^3$. The result for $k=2$ follows.

Suppose  in the following that $k\geq3$.

Suppose that $\Delta=2$.
Since $T\ncong P_{n,k}$, there is at least one edge with at least three vertices of degree $2$.
Suppose that there are at least two such edges. Let $u$ be a vertex of degree one in $T$.
Choose an edge $e=\{w_1,\dots ,w_k\}$ in $E(T)$ with at least $3$ vertices of degree two such that $d_T(u,w_1)$ is as large as possible, where  $d_T(u,w_1)=d_T(u,w_i)-1$ for $2\leq i\leq k$.
Then
there are at least two pendant paths  at different vertices of $e$, say $P=(w_i,e_1,u_1,\dots, u_{p-1},e_p,u_p)$ at $w_i$ and $Q=(w_j,e'_1,v_1,\dots,v_{q-1},e'_q,v_q)$ at $w_j$, where $1\leq i<j\leq k$ and $p,q\geq 1$.
Then $T\cong H_{w_i,w_j}(p,q)$, where $H=T[V(T)\setminus (V(P\cup Q)\setminus\{w_i,w_j\}) ]$. Assume that $p\ge q$.
Note that $T'=H_{w_i,w_j}(p+1,q-1)$ is a $k$-uniform hypertree that is not isomorphic to $P_{n,k}$.
By Corollary~\ref{moving22}, we have $\rho(T)<\rho(T')$, a contradiction.
Thus there is exactly one edge $e$ with at least three vertex of degree $2$.

We claim that there are exactly three vertices of degree two in $e$. Otherwise, $k\ge 4$ and
 there are four vertices $w_1,w_2,w_3$ and $w_4$ of degree two in $e$.
Let $Q_i$  be the pendant path of length $l_i$  at $w_i$, where $l_i\ge 1$ for $i=1,2,3,4$. Assume that $l_1\ge l_2$. Let $G=T[V(T)\setminus (V(Q_1\cup Q_2)\setminus \{w_1,w_2\})]$.
Then $T\cong G_{w_1,w_2}(l_1,l_2)$,
Note that $T''=G_{w_1,w_2}(l_1+1,l_2-1)$ is a $k$-uniform hypertree that is not isomorphic to $P_{n,k}$.
By Corollary~\ref{moving22}, $\rho(T)<\rho(T'')$, a contradiction.
Thus there are exactly three vertices of degree two in $e$, say $w_1,w_2,$ and $w_3$.

Let $Q_i$ be the pendant path at $w_i$ with length $l_i$, where $i=1,2,3$ and $l_i\ge1 $. Assume that $l_1\geq l_2\geq 2$. Let $G=T[V(T)\setminus (V(Q_1\cup Q_2)\setminus\{w_1,w_2\})]$.
Then $T\cong G_{w_1,w_2}(l_1,l_2)$. Note that
 $T^*=G_{w_1,w_2}(l_1+1,l_2-1)$ is a $k$-uniform hypertree that is not isomorphic to $P_{n,k}$.
By Corollary~\ref{moving22}, $\rho(T)<\rho(T^*)$, a contradiction.
Thus there are at least two of $Q_1,Q_2$ and $Q_3$ with length $1$.
Thus $T\cong F_{n,k}$.

By Lemma~\ref{hyper1}, $\rho(B_{n,k}^3)<\rho(F_{n,k})$. Thus $T\cong F_{n,k}$.
\end{proof}

\begin{Theorem}\label{min-hyper1}
Let $T$ be a $k$-uniform hypertree on $n$ vertices, where $\frac{n-1}{k-1}\geq1$. Then $\rho(T)\geq \rho(S_{n,k})$ with equality if and only if $T\cong S_{n,k}$.
\end{Theorem}

\begin{proof}  It is trivial if $\frac{n-1}{k-1}\leq2$. Suppose that $\frac{n-1}{k-1}\geq3$.
Let $T$ be a $k$-uniform hypertree of order $n$ with minimum distance spectral radius.

Let $d$ be the diameter of $T$. Obviously, $d\ge 2$.  Suppose that $d\geq 3$.
Let $P=(v_0, e_1,v_1,\dots,v_{d-1},e_d,v_d)$ be a diametral path of $T$.
Let $e_{d-1}=\{w_1,w_2,\dots, w_k\}$, where $w_{k-1}=v_{d-1}$ and $w_k=v_{d-2}$.
For $1\leq i\leq k-1$, let $t_i$ be the number of pendant edges at $w_i$  outside $P$,  where  $t_{k-1}\geq 1$.
For $1\leq i\leq k-1$ with $t_i\ge 1$, let $e_i^1,\dots,e_i^{t_i}$ be the $t_i$ pendant edges at $w_i$  outside $P$, where  $e_{w_{k-1}}^1=e_s$. Then $T\cong G_{e_{d-1},0}(t_1,\dots ,t_{k-1})$, where $G=T[V(T)\setminus E_1]$ and $E_1=\cup _{1\leq i\leq k-1 \atop 1\leq j \leq t_i} (e_i^j\setminus\{w_i\})$.
Let $T'$ be the hypergraph obtained from $G_{e_{d-1},0}(t_1,\dots ,t_{k-1})$ by moving edges $e_i^1,\dots,e_i^{t_i}$ from $w_i$ to $w_k$ for all $i$ with $1\leq i\leq k-1$ and $t_i\ge 1$.
Then $T'\cong G_{e_{d-1},k-1}(t_1, \dots, t_k)$. By Corollary~\ref{moving33}, $\rho (T)>\rho(T')$, a contradiction.
Thus $d=2$, implying that $T\cong S_{n,k}$.
\end{proof}

An automorphism of a hypergraph $G$ is a bijection  on $V(G)$ 
which induces a bijection  on $E(G)$.

\begin{Lemma}\label{hyper3} Let $G$ be a connected $k$-uniform hypergraph with $\eta$ being an automorphism of $G$. Let $x=x(G)$. Then $\eta(u)=v$ implies that $x_u=x_v$.
\end{Lemma}

\begin{proof}
Let $P$ be the permutation matrix that corresponds to the automorphism $\eta$ of $G$. Then $D(G)=P^{\top}D(G)P$. Since $\rho(G)x=D(G)x$, we have $\rho(G)=x^{\top}Dx=(Px)^{\top}D(G)(Px)$. Obviously, $Px$ is a positive unit vector. 
Thus $Px=x$, from which the result follows.
\end{proof}

For $\frac{n-1}{k-1}\geq 3 $ and $1\leq a\leq \left\lfloor\frac{n-k}{2(k-1)}\right\rfloor$, let $D_{n,a}$ be the $k$-uniform hypertree obtained from vertex-disjoint $S_{a(k-1)+1,k}$ with center $u$ and $S_{n-k-a(k-1)+1,k}$ with center $v$  by adding $k-2$ new vertices $w_1,\dots,w_{k-2}$ and an edge $\{u,v, w_1,\dots,w_{k-2}\}$.

\begin{Lemma}\label{hyper2}
For $2\leq a\leq \left\lfloor\frac{n-k}{2(k-1)}\right\rfloor$, $\rho(D_{n,a})>\rho(D_{n,a-1})$.
\end{Lemma}

\begin{proof}
Let $b=\frac{n-k}{k-1}-a$.
Let $u$ and $ v $ be the vertices of $D_{n,a}$ with degree $a$ and $b$, respectively.
Let $E(D_{n,a})=\{e_i:  1\leq i\leq a+b+1\}$ and  $e_i=\{w_i^1,\dots,w_i^k\}$ for $1\leq i\leq a+b+1$, where $w_i^k=u $ if $1\leq i\leq a$,  $w_i^k=v $ for $a+1\leq i\leq a+b$, and $w_i^1=u$ and $w_i^k=v$ for $i=a+b+1$. Let $x=x(D_{n,a})$.
By Lemma~\ref{hyper3}, $x_{w_i^1}=x_{w_i^j}$ and $x_{w_1^1}=\cdots =x_{w_a^{1}}$ for $1\leq i\leq a$ and $1\leq j\leq k-1$,
$x_{w_i^1}=x_{w_i^j}$ and $x_{w_{a+1}^1}=\cdots =x_{w_{a+b}^{1}}$ for $a+1\leq i\leq a+b$ and $1\leq j\leq k-1$,
and $x_{w_i^2}=\cdots =x_{w_i^{k-1}}$ for $i=a+b+1$.
Let $\alpha=x_{w_1^1}$, $\beta=x_{w_{a+1}^1}$, and $\gamma=x_{w_{a+b+1}^2}$.
From the eigenequations of $D_{n,a}$ at $w_1^1$, $w_{a+1}^1$,  $w_{a+b+1}^1$, $u$ and $v$,
we have
\begin{eqnarray*}
\rho(D_{n,a})\alpha&=& (2(k-1)a-k)\alpha+3(k-1)b\beta +2(k-2)\gamma +x_u+2x_v ,\\
\rho(D_{n,a})\beta &=& 3(k-1)a\alpha+(2(k-1)b-k)\beta +2(k-2)\gamma +2x_u+x_v ,\\
\rho(D_{n,a})\gamma &=& 2(k-1)a\alpha+2(k-1)b\beta +(k-3)\gamma +x_u+x_v ,\\
\rho(D_{n,a})x_u &=& (k-1)a\alpha+2(k-1)b\beta +(k-2)\gamma +x_v ,\\
\rho(D_{n,a})x_v &=& 2(k-1)a\alpha+(k-1)b\beta +(k-2)\gamma +x_u .
\end{eqnarray*}
We view these equations as a homogeneous linear system  in the five  variables $\alpha, \beta, \gamma, x_u$, and $x_v$. Since it has a  nontrivial solution,  we have
\[
\det\left(
\begin{matrix}
 2(k-1)a-k-\rho & 3(k-1)b & 2(k-2) & 1 & 2\\
3(k-1)a & 2(k-1)b-k-\rho & 2(k-2) & 2 & 1\\
 2(k-1)a & 2(k-1)b & k-3-\rho & 1 & 1\\
 (k-1)a & 2(k-1)b & k-2 &  -\rho & 1,\\
 2(k-1)a & (k-1)b & k-2 &  1  & -\rho
\end{matrix}
\right)=0,
\]
where $\rho=\rho(D_{n,a})$.
Thus $\rho(D_{n,a})$ is the largest root of the equation $g_a(t)=0$, where
\begin{eqnarray*}
g_a(t)&=&-t^5+t^4(2ak+2bk-k-2a-2b-3)\\
&&+t^3(k^2+4ak^2+4bk^2+5abk^2-10abk\\
&&-ak-bk-4k+5ab-3a-3b-3)\\
&&+t^2(k^3+3abk^3+2ak^3+2bk^3+k^2-3abk^2\\
&&+4ak^2+4bk^2-3ab k-5ak-5bk-5k+3ab-a-b-1)\\
&&+t(2k^3+2ab k^3+3ak^3+3bk^3-k^2-4ab k^2\\
&&-ak^2-bk^2+2abk-2ak-2bk-2k)\\
&&+k^3+ak^3+bk^3-k^2-ak^2-bk^2.
\end{eqnarray*}
For $2\leq a\leq \left\lfloor\frac{n-k}{2(k-1)}\right\rfloor$, it is easily seen that
\[g_a(t)-g_{a-1}(t)=
(b+1-a)(k-1)^2 t[5t^2+(3k+3)t+2k].\]
Let $\rho_i=\rho(D_{n,i})$ for $1\leq i\leq \left\lfloor\frac{n-k}{2(k-1)}\right\rfloor$.
Then \begin{eqnarray*}
g_a(\rho_{a-1})&=&g_a(\rho_{a-1})-g_{a-1}(\rho_{a-1})\\
&=&
(b+1-a)(k-1)^2 \rho_{a-1}[5\rho_{a-1}^2+(3k+3)\rho_{a-1}+2k]\\
&>&0,
\end{eqnarray*}
from which, together with the fact that $g_a(t)<0$ for $t>\rho_a$, we have $\rho_a>\rho_{a-1}$ for $2\leq a\leq \left\lfloor\frac{n-k}{2(k-1)}\right\rfloor$.
\end{proof}

\begin{Theorem}\label{min-hyper2}
Let $T$ be a $k$-uniform hypertree of order $n$, where $T\ncong S_{n,k}$ and $\frac{n-1}{k-1}\geq3$. Then $\rho(T)\geq \rho(D_{n,1})$ with equality if and only if $T\cong D_{n,1}$.
\end{Theorem}
\begin{proof} It is trivial if $\frac{n-1}{k-1}=3$. Suppose that $\frac{n-1}{k-1}\geq 4$.
Let $T$ be a $k$-uniform hypertree of order $n$  nonisomorphic to $S_{n,k}$ with minimum distance spectral radius.

Let $d$ be the diameter of $T$. Since  $T\ncong S_{n,k}$, we have $d\ge 3$.
Let $P=(v_0, e_1,v_1,\dots,v_{d-1},e_d,v_d)$ be a diametral path of $T$.
Let $e_{d-1}=\{w_1,\dots, w_k\}$, where $w_{k-1}=v_{d-1}$ and $w_k=v_{d-2}$.
For $1\leq i\leq k-1$, let $t_i$ be the number of the pendant edges at $w_i$  outside $P$, where  $t_{k-1}\geq 1$.
For   $1\leq i\leq k-1$ with $t_i\ge 1$, let $e_i^1,\dots,e_i^{t_i}$ be the $t_i$  pendant edges at $w_i$  outside $P$, where  $e_{w_{k-1}}^1=e_d$.
Then  $T\cong G_{e_{d-1},0}(t_1,\dots ,t_{k-1})$,  where $G=T[V(T)\setminus E_1]$ and $E_1=\cup _{1\leq i\leq k-1 \atop 1\leq j \leq t_i}(e_i^j\setminus\{w_i\})$.
Assume that $T= G_{e_{d-1},0}(t_1,\dots ,t_{k-1})$.

Suppose that $d\geq 4$.
Let $T'$ be the hypergraph obtained from $T$ by moving edges $e_i^1,\dots,e_i^{t_i}$  from $w_i$ to $w_k$ for all $i$ with $1\leq i\leq k-1$ and $t_i\ge 1$.
Then $T'\cong G_{e_{s-1},k-1}(t_1, \dots, t_{k-1})\ncong S_{n,k}$. By Corollary~\ref{moving33}, $\rho (T)>\rho(T')$, a contradiction.
Thus $d=3$, implying that
$T=G_{e_2,0}(t_1,\dots, t_{k-1})$.

Suppose that $k\geq 3$ and  $t_i\geq 1$ for some $i$ with $1\leq i\leq k-2$.
Let $T''$ be the hypergraph obtained from $T$ by moving edges $e_i^1,\dots,e_i^{t_i}$ from $w_i$ to $w_k$ for all $i$ with $1\leq i\leq k-2$ and $t_i\ge 1$.
Then $T''\cong G_{e_2, k-2}(t_1, \dots,t_{k-1})\ncong S_{n,k}$.
By Corollary~\ref{moving33}, $\rho (T)>\rho(T'')$, a contradiction.
Thus $t_i=0$ for each $1\leq i\leq k-2$.
It follows that $T=G_{e_2,0}(0,\dots, 0, t_{k-1})\cong D_{n,a}$, where $a=\min \{t_{k-1}, \frac{n-k}{k-1}-t_{k-1}\}$ for $k\ge 3$.
Obviously, this is also true for $k=2$. 
By Lemma~\ref{hyper2}, we have $T\cong D_{n,1}$
\end{proof}

\section{Concluding remarks}

We propose three types of graft transformations for a $k$-uniform hypergraph, and  study the effect of them to increase or decrease the distance spectral radius. We show that $P_{n,k}$ and $S_{n,k}$ are  the unique  $k$-uniform hypertrees with  maximum and minimum distance spectral radius, respectively, and we also determine the unique  $k$-uniform hypertrees with second maximum and second minimum distance spectral radius, respectively. Besides, we determine the unique hypertree with maximum distance spectral radius among $k$-uniform hypertrees with given maximum degree.
Some theorems in this paper echo results on the distance spectral radius of ordinary graphs in the literature, see, e.g. \cite{SuPo90, SI,WZ,?}. However, there are differences between $k$-uniform hypergraphs with $k\ge 3$ and ordinary graphs ($k=2$). In Theorem \ref{max-hyper3}, we show that $F_{n,k}$ is the unique hypertree with second maximum distance spectral radius among $k$-uniform hypertrees of order $n$. It is easy to see that $F_{n,k}$ for $k\ge 3$ and $F_{n,2}$ have quite different structure.

\bigskip
%


\end{document}